\documentclass[12pt,reqno,twoside]{amsart}

\usepackage{amsmath}

\usepackage{amssymb}
\usepackage[all]{xy}
\usepackage{epsfig}
\usepackage{color}
\title[Large cusps and short geodesics]{Finiteness
results for flat surfaces: large cusps and short geodesics}
\author{John Smillie
}
\address{Cornell University, Ithaca, NY {\tt smillie@math.cornell.edu}}
\author{Barak Weiss}
\address{Ben Gurion University, Be'er Sheva, Israel 84105
{\tt barakw@math.bgu.ac.il}}

\font\sb = cmbx8 scaled \magstep0

\font\sn = cmssi8 scaled \magstep0

\long\def\combarak#1{\ifdraft{\sb #1 }\else\ignorespaces\fi}

\newif\ifdraft\drafttrue

\draftfalse

\newcommand\name[1]{\label{#1}{\ifdraft{\sn [#1]}\else\ignorespaces\fi}}

\newcommand\eq[2]{{\ifdraft{\ \tt
[#1]}\else\ignorespaces\fi}\begin{equation}\label{#1}{#2}\end{equation}}

\newcommand {\equ}[1]     {\eqref{#1}}

\newcommand{\MM}{{\mathcal{H}}}

\newcommand{\Q}{{\mathbb {Q}}}
\newcommand{\HH}{{\mathbb{H}}}

\newcommand{\R}{{\mathbb{R}}}

\newcommand{\ii}{{\mathbf{i}}}
\newcommand{\Z}{{\mathbb{Z}}}

\newcommand{\N}{{\mathbb{N}}}

\newcommand{\SL}{\operatorname{SL}}
\newcommand{\PSL}{\operatorname{PSL}}

\newcommand{\diag}{{\rm diag}}

\newcommand {\ignore}[1]  {}

\newcommand{\interior}{{\rm int}}

\newcommand{\LL}{{\mathcal L}}

\newcommand{\til}{\widetilde}

\newcommand{\Mat}{{\operatorname{Mat}}}

\newcommand{\sm}{\smallsetminus}

\newcommand{\NSMP}{{\mathrm{SMP}}}
\newcommand{\Hyp}{{\mathrm{Hyp}}}
\newcommand{\LCM}{{\mathrm{LCM}}}

\newcommand{\NLC}{{\mathrm{SC}}}
\newcommand{\Aff}{{\mathrm{Aff}}}
\newcommand{\SC}{{\mathrm{SC}}}

\newtheorem{thm}{Theorem}[section]

\newtheorem{lem}[thm]{Lemma}

\newtheorem{prop}[thm]{Proposition}

\newtheorem{cor}[thm]{Corollary}

\newtheorem{remark}[thm]{Remark}

\begin{document}
\maketitle

\begin{abstract}
For fixed $g$ and $T$ we show the finiteness of the set of affine
equivalence classes of flat surfaces of genus $g$ whose Veech group contains a 
cusp of hyperbolic co-area less than $T$.
We obtain new
restrictions on Veech groups: we show that any non-elementary Fuchsian
group can appear only finitely many times in a fixed stratum, that
any non-elementary Veech group is of finite index in its normalizer,
and that the quotient of $\HH$ by a non-lattice Veech group admits
arbitrarily large embedded disks. A key ingredient of the proof is the
finiteness of the set of affine equivalence classes of flat surfaces
of genus $g$ whose Veech group contains a hyperbolic element with
eigenvalue less than $T$. 
\end{abstract}

\section{Introduction}
Our objects of study are flat surfaces and
their affine automorphism groups. These structures arise in the study 
of rational polygonal billiards, in Thurston's classification of
surface diffeomorphisms in connection with measured foliations, and in
complex analysis. The class of flat surfaces is subdivided into
translation surfaces and half-translation surfaces, which correspond
in the complex analysis literature to abelian and quadratic
differentials respectively.

Let $\Aff(M)$ denote the affine automorphism group of a flat surface $M$. 
Let $G=\SL(2,\R)$ if $M$ is a translation surface, and $\PSL(2,\R)$ if
$M$ is a half-translation surface. Taking the differential of an
automorphism yields a 
homomorphism $D: \Aff(M) \to G$ with finite kernel, whose image 
$\Gamma_M$ is called the {\em Veech group} of $M$. Alternatively,
$\Gamma_M$ is the stabilizer, under the $G$-action, of $M$ in the
stratum $\MM$ of flat surfaces containing $M$. 
For a typical flat surface $\Aff(M)$ (and hence $\Gamma_M$) is
trivial; however surfaces with non-trivial Veech groups are quite
interesting. Much of the interest in the subject was generated by
Veech's discovery \cite{Veech - alternative} of Veech groups which are
non-arithmetic lattices (recall that $\Gamma$ is called a {\em
lattice} if $\HH/\Gamma$ has finite area). Many additional Veech
groups have been described, including infinitely generated ones
\cite{HS, McM}, and there are many intriguing questions regarding
groups which may arise as Veech groups. See \cite{HSsurvey}
for a recent survey of the field, and \cite{toronto} for a problem
list. 

A fundamental question in this regard is to understand the
commensurability classes of Fuchsian groups which contain Veech
groups 
(recall that a {\em Fuchsian group} is a discrete subgroup of $G$ and
two of such groups are {\em commensurable}
if their intersection is of finite index in both).
We will describe some restrictions on groups commensurable to a Veech group
$\Gamma$ and on the geometry of the corresponding $\HH/\Gamma$. The
question of which groups within a given commensurability class
actually arise as Veech groups is also of considerable interest, see
\cite{HS2}. 

It is clear that if $\Gamma$ is a lattice, there is an upper limit on
the radius of a maximal embedded disk. We have the following converse,
which provides a new characterization of lattice surfaces. 
\begin{thm}
\name{cor: new}
If $\Gamma$ is commensurable to a Veech group and is not a lattice
then for every $R$, $\HH/\Gamma$ contains an embedded disk of radius
$R$. 
\end{thm}

We derive a purely group-theoretic property of Veech groups:
\begin{thm}
\name{cor: normalizer}
A non-elementary Fuchsian group commensurable to a Veech group is of finite index
in its normalizer. 
\end{thm}

These theorems follow from some finiteness results which involve
an upper bound
on either the covolume of a cusp for $\Gamma_M$, or the eigenvalue of a
hyperbolic element in $\Gamma_M$. 
For a Fuchsian group $\Gamma$, to a cusp in $\HH/\Gamma$ one may
associate the {\em cusp area} ---  the hyperbolic area of a
maximal continuous embedded family of parallel horocycles --- and the
number of cylinders in the corresponding decomposition of the
surface. Moreover these quantities are invariant under affine equivalence. 
Equivalently, in group-theoretic terms, a cusp in
$\HH/\Gamma$ is determined by the conjugacy class of a 
maximal parabolic subgroup $P \subset \Gamma$. Associating to each
such $P$ the area $t_0(\Gamma, P)$ and cylinder number $m =
m(M, P) \in \N$, we have that for all $g \in G$, $\Gamma_{gM} =
g\Gamma_{M} g^{-1}$ contains the maximal 
parabolic subgroup $gPg^{-1}$ and 
$$t_0(g\Gamma_Mg^{-1}, gPg^{-1}) = t_0(\Gamma_M, P), \ \ \ m(gM, gPg^{-1}) = m(M,P).$$

Let 
$$\NLC(m, T) = \{ (M, P): 
t_0(\Gamma_M, P) \leq 
T, \, m(M,P) =m \},$$
where $M$ ranges over all flat surfaces, and
$P$ ranges over all maximal parabolic subgroups of $\Gamma_M$
(SC stands for `small cusp'). We denote by $\til \NLC (m, T)$ the
corresponding set of affine equivalence classes.

The following holds:

\begin{thm}
\name{thm: cusp areas finite}
For any $T>0$ and any $m \in \N$, $\til \NLC(m, T)$ is finite.
\end{thm}

Our proof yields an explicit bound on $\# \, \til \NLC(m,T)$, 
see Theorem
\ref{thm: NLC precise}.

Let $\mu$ be haar measure on $G$ and let $\Gamma$ be a discrete
subgroup of $G$. We denote the covolume of $\Gamma$ in $G$ by
$\bar{\mu}(\Gamma)$; this number depends only on the conjugacy class
of $\Gamma$. We
deduce:

\begin{cor}
\name{cor: covolumes finite}
Let $\MM$ be a stratum of flat surfaces and let $R, T>0$. Then the
following sets are finite:
\begin{itemize}
\item[(i)]
Affine equivalence classes of $M \in \MM$ for which $\HH/\Gamma_M$
contains no embedded ball of radius $R$. 
\item[(ii)]
Affine equivalence classes of $M \in \MM$ for 
which $\bar{\mu}(\Gamma_M) <T$.
\end{itemize}
\end{cor}
Note that Veech 
\cite{Veech - alternative} constructs lattice surfaces $M_n$ for
all $n \neq 4$ (on different strata) such that
$\bar{\mu}(\Gamma_{M_n}) < 2\pi.$ Thus one cannot omit the hypothesis that
$M$ is contained in a fixed stratum in 
Corollary \ref{cor: covolumes finite}. Assertion (ii) was proved
independently by Curt 
McMullen, using the algebraic geometry of moduli space. 

Suppose $\Gamma_M$ contains
a hyperbolic element $h$. We denote the larger eigenvalue of $h$ by $\lambda(h)$
and call it the {\em eigenvalue} of $(M,h)$. Also, associated to $h$
are {\em Markov partitions} of $M$ (see \S5). We let
$p = p(M,h)$ be the minimal number of parallelograms 
in a Markov partition. For $T>0$ and $m \in \N$ we define 
$$\NSMP (p,T) = \{(M,h): p(M,h) =p, \lambda(h) <T \}$$
(SMP stands for `simple Markov partition'). 
As before these quantities are invariant under affine equivalence: if
$h \in \Gamma_M$ is hyperbolic then $ghg^{-1} \in 
\Gamma_{gM}$ is also hyperbolic with $p(M,h)=p(gM, ghg^{-1})$ and $\lambda(h)
= \lambda(ghg^{-1})$. We denote
the set of affine equivalence classes in $\NSMP$ by $\til \NSMP.$ 
Repeating a folklore argument\footnote{The argument is probably due to
Thurston. We were unable to find a suitable
reference, but see \cite{Thurston} for a hint and \cite{Rykken} for
more details.} we obtain:
\begin{prop}[Thurston, Veech]
\name{prop: Thurston bound}
For a fixed $T>0$ and $p \in \N$, 
$\til \NSMP(p,T)$ is finite.
\end{prop}
The existence of Markov partitions is sketched in \cite{travaux} and
we explain it in detail in 
an appendix to our paper. In particular we show that bounding $p$
is equivalent to bounding the genus of $M$. 
Note that by \cite{Penner}, there are pairs $(M_n,
h_n)$ with $\lambda(h_n) \to 1$, i.e. one cannot omit
a bound on $p$ (or on the topology of $M$) from the statement. 

Recall that the {\em geodesic flow} is the restriction of the
$G$-action on $\MM$ to the one-parameter subgroup of diagonal
matrices. There is a bijective correspondence between $G$-orbits of
pairs $(M,h)$ as above,
and surfaces with a periodic trajectory under the geodesic
flow, with the length of the corresponding trajectory equal to $\log
\lambda(h)$. 
Thus Proposition \ref{prop:
Thurston bound} is equivalent to the statement, proved by Veech \cite{Veech
- geodesic flow}, that the number of
periodic geodesic trajectories in $\MM$ of length at most $T$
is finite.

For a group $\Gamma$ and $h
\in \Gamma$ we write $h^{\Gamma}$ for the conjugacy class of $h$ in
$\Gamma$. From Theorem \ref{thm: cusp areas finite} and Proposition
\ref{prop: Thurston bound} we derive a 
restriction on Veech groups:


\begin{cor}
\name{cor: restriction on Gamma}
Suppose $\Gamma$ is commensurable to a Veech group. Then for any
$T>0$, the following sets are finite:
\begin{itemize}
\item[(I)]
$\{ P \subset \Gamma : P \mathrm{\ is \ a \ maximal \ parabolic,
} \ t_0(\Gamma,P) \leq T \}.$
\item[(II)]
$\{h^{\Gamma} : h \in \Gamma \mathrm{\ is \ hyperbolic \ with \ } \lambda(h)
< T \}.$
\item[(III)]
$\{f^{\Gamma}: f \in \Gamma \mathrm{\ is \ elliptic \ with \ cone \
area \ at \ most \ }T \}.$
\end{itemize}
\end{cor}
Here the {\em cone area} associated with an elliptic $f \in \Gamma$ is
the area of $B(z_f, R)/ \langle f \rangle$, where $z_f \in \HH$ is the
fixed point of $f$ and 
\eq{eq: defn R}{
R = R(f)= \sup \left \{r: B(z_f, r) / \langle f \rangle \to \HH/\Gamma \mathrm{
\ is \ injective} \right \}.
}


Given a Fuchsian group, it is natural to ask `how often' it arises as a
Veech group. Cyclic parabolic subgroups are associated with cylinder
decompositions and are easily described. In any stratum there are
infinitely many of them belonging to different $G$-orbits. Using torus
covers one can construct 
infinitely many flat surfaces $M$, in different strata, with the same
non-elementary Veech group. Also, if $\Gamma = \Gamma_M$ and $g \in G$
normalizes $\Gamma$ then also $\Gamma = \Gamma_{gM}$. As an
application of our results we show that aside from these
simple constructions, each group can only appear finitely many times, namely:

\begin{cor}
\name{cor: group determines surface}
For any stratum $\MM$ and any infinite Fuchsian group $\Gamma$ which
is not cyclic parabolic, the
set 
\eq{eq: M(Gamma)}{\left\{M \in \MM: \Gamma_M =\Gamma \right\}
}
contains finitely many $N$-orbits, where $N$ is the normalizer of
$\Gamma$ in $G$. In particular if $\Gamma$ is non-elementary, 
\equ{eq: M(Gamma)} is finite. 
\end{cor}




\medskip
{\bf Acknowledgements.} We thank Yair Minsky and Yair Glasner for useful
discussions. The support of NSF grant 
DMS-0302357, BSF grant 2004149 and ISF grant 584-04 is gratefully
acknowledged. Some of the results of this paper were announced in
\cite{toronto}.

\section{Basics}
In this section we set some
notation, and collect standard results. 
\subsection{Flat surfaces}
\name{section: flat
surfaces}
We begin by listing some definitions. For more details we refer the reader to 
\cite{MT, Vorobets, zorich}.

Throughout this paper, $S$ denotes a compact connected orientable
surface of genus $g$.
When $S$ is equipped with the structure of a flat surface or a
quadratic differential we will usually denote it by $M$. When
confusion may arise we will also use $M$ to denote the underlying
surface $S$. A flat surface admits several equivalent
definitions. It may be thought of as an equivalence class of atlases
of charts $(U_{\alpha}, \varphi_{\alpha})$ covering all but a finite
set $\Sigma = \Sigma_M \subset S$ of {\em singularities}, such that
the transition functions 
$\varphi_{\alpha} \circ \varphi_{\beta}^{-1}$ are of the form 
$z \mapsto \pm z + c$, and such that for each $\sigma \in \Sigma$ the charts combine to form
$k$-pronged singularity at $\sigma$, where $k=k_\sigma \in \N$. If
$k_\sigma=2$ then $\sigma$ is called a {\em removable singularity} or
{\em marked point.}
We always assume that $\Sigma_M \neq \varnothing.$
Two atlases $M, M'$ are {\em compatible} if $M \cup M'$ is also a
quadratic differential. The atlases $M=\left(U_{\alpha}, \varphi_{\alpha} \right), \,
M'$ are {\em equivalent} if there is a
self-homeomorphism $h : S \to S$ such that $h(\Sigma_M) = \Sigma_{M'}$
and $\left(h(U_{\alpha}),
\varphi_{\alpha} \circ h \right)$ is compatible with $M'$. A flat
surface is called a {\em translation surface}, or an {\em abelian
differential}, if there is a compatible atlas in which all transition 
functions are of the form $z \mapsto z+c.$ A flat surface which is not
a translation surface is called a {\em half-translation surface}. 
An {\em
affine automorphism} of $M$ is a  
self-homeomorphism of $S$ which is affine in each chart.
Fix the data giving the number of singularities, the vector
$\vec{k}=\left(k_\sigma \right)_{\sigma \in \Sigma}$, and the determination
whether or not the flat surfaces are translation surfaces;  then the set of all
quadratic differentials sharing this data is called a {\em
stratum}.
Each stratum is equipped with a structure of
an affine orbifold, which is locally modelled on a relative cohomology
group, see \cite{MS, MT}. The group $G$ acts on $\MM$ by
post-composition on each chart in an atlas.  

A flat surface inherits from the plane a singular foliation on $S$ called the
{\em horizontal (resp. vertical) foliation}, with 
each chart foliated into the lines parallel to the
x-axis (resp. y-axis). 
A {\em saddle connection} is a straight segment joining singularities
or punctures,
with no singularities or punctures in its interior. The set of all
saddle connections in direction $\theta$ is denoted by $\LL_M(\theta)$. 

Throughout this paper we will identify elements of $G$ with matrices
of $\SL(2,\R)$. We will need the
following four one-parameter subgroups of $G$: 
\[
\begin{split}
g_t & = \left( \begin{array}{cc}
e^{t/2} & 0 \\ 0 & e^{-t/2} \end{array} \right), \ \ \ r_{\theta}=
\left(\begin{array}{cc}
\cos \theta & -\sin \theta \\
\sin \theta & \cos \theta
\end{array}
\right),\\
h_s & = \left(\begin{array}{cc} 1 & s \\ 0 & 1 
\end{array}
\right),
 \ \ \ \ \ \ \ \  \ \, \ \til h_s = \left(\begin{array}{cc} 1 & 0 \\ s & 1 
\end{array}
\right).
\end{split}
\]
We say that $\lambda_1, \ldots, \lambda_k \in \R$ are {\em commensurable} if
$\lambda_i/\lambda_j \in \Q$ for all $i,j \in \{1, \ldots, k\}$. If this
holds, we denote by $\LCM(\lambda_1, \ldots, \lambda_k)$ the smallest
positive number which is an integer multiple of all the $\lambda_i$.

\subsection{Cylinder decompositions, gluing patterns}
A {\em cylinder} for $M$ is a topological annulus which is
isometric to $\R /w
\Z \times (0,h)$. It is {\em maximal} if it is not contained in a
larger cylinder, and this implies that both of its boundary components
in $S$ contain singularities. The 
{\em height, width}, and {\em inverse modulus} of the
cylinder are $h, w$, and $w/h$ respectively, and a curve which wraps around the cylinder
parallel to its boundary is called a {\em waist curve}. 
A {\em cylinder decomposition} is a decomposition of $M$ into maximal
cylinders with disjoint interiors. The waist curves of
all the cylinders in a cylinder decomposition are parallel. Two
cylinder decompositions of $M$ are called {\em transverse} if the
directions of waist curves in each decomposition are different. 

\subsubsection{Gluing pattern for a pair of cylinder decompositions}
Suppose for $i=1,2$ that we have two transverse
cylinder decompositions $M = C^{(i)}_1 \cup \cdots \cup
C^{(i)}_{m_i}$. For each $i \in \{1, \ldots, m_1\}, j \in \{1, \ldots,
m_2\},$ the intersection $C^{(1)}_i \cap C^{(2)}_j$ consists of a finite union of
parallelograms. The resulting parallelograms have disjoint interiors
and boundary identifications. To describe how the parallelograms are
attached to each other we follow 
ideas of \cite{EO1}.

Suppose first that $M$ is a translation surface. In this case one can
consistently label the edges of the
parallelograms with labels right, left, top, bottom. Label the
parallelograms by $1, \ldots, \ell$, let $S_{\ell}$ denote the
group of permutations on $\ell$ elements, and define $\sigma_1, \sigma_2
\in S_\ell$, where $\sigma_1(k_1)=k_2$ (resp. $\sigma_2(k_1)=k_2$) if the
right (resp. top) edge of the $k_1$th parallelogram is attached to
the left (resp. bottom) edge of the $k_2$th parallelogram. 

Replacing the labelling of the parallelograms amounts to replacing
$\sigma_1, \sigma_2$ with $\tau \sigma_1 \tau^{-1}, \tau \sigma_2
\tau^{-1}$ for some $\tau \in S_\ell$. We denote the equivalence class
of $(\sigma_1, \sigma_2) \in S_\ell \times S_\ell$ under simultaneous conjugations by
$[(\sigma_1, \sigma_2)]$, and call $[(\sigma_1, \sigma_2)]$ the {\em
gluing pattern} of $M$ and the given cylinder decompositions.
From the
gluing pattern it is simple to recover the number of cylinders in each
decomposition, the number of connected components in each
intersection of cylinders, and the stratum to which $M$
belongs. In particular the following hold:
\begin{itemize}
\item[(i)]
For $i=1,2$, $\sigma_i$ has $m_i$ cycles, and the length of the
$j$th cycle of $\sigma_1$ (resp. $\sigma_2$) is $\sum_i a_{ij}$ (resp.
$\sum_i a_{ji}$).
\item[(ii)]
The subgroup of $S_\ell$ generated by $\sigma_1, \sigma_2 $ acts
transitively on $\{1, \ldots, 
\ell \}.$ 
\end{itemize}

Moreover 
$M$ is completely determined by the dimensions and
orientation of the parallelogram and the corresponding gluing pattern;
indeed such information gives rise to an explicit atlas of
charts as in \S2.1. 

In case $M$ is a half-translation surface one can define a gluing pattern in a
more complicated way by subdividing each rectangle into four
`quarter-tiles'. 
This will not
be used in the current paper.

\ignore{
In case $M$ is a general flat surface (i.e. not necessarily a
translation surface), one subdivides each
parallelogram into four `quarter-tiles' and defines the gluing pattern
on the resulting $4\ell$ parallelograms using 4 involutions 
$\sigma_i \in S_{4\ell}, \, i=1, \ldots, 4$. The permutations
correspond respectively to the operations of attaching the
parallelograms horizontally, switching their right and left sides,
switching their top and bottom sides, and attaching them vertically. They satisfy 
\begin{itemize}
\item[(i')]
$\sigma_1 \sigma_2$ (resp. $\sigma_3 \sigma_4$) has $2m_i$ cycles, and
the cycle lengths of $\sigma_1 \sigma_2$ (resp. $\sigma_3 \sigma_4$)
is record the column sums $\sum_i a_{ij}$ (resp.
row sums $\sum_i a_{ji}$).
\item[(ii')]
The subgroup of $S_\ell$ generated by the $\sigma_i$ acts
transitively on $\{1, \ldots, 
4\ell \}.$ 
\end{itemize}
We denote by $[(\sigma_1, \sigma_2, \sigma_3, \sigma_4)]$ the
equivalence class of such involutions, under relabelling of the
parallelograms and switching labels of sides in each.
}

\subsubsection{Gluing pattern for parallelograms}
We will also have occasion to consider flat surfaces made up of
finitely many parallelograms, glued along edges. We suppose that $M$ is a
translation surface which is a union of
closed metric parallelograms $P_1, \ldots, P_m$, with disjoint interiors,
such that the two directions of parallel sides are
the same for all $P_i$. For ease of notation we assume that these
sides are horizontal and vertical. The interiors of the parallelograms
are isometrically embedded, so the singularities of $M$ lie on the
boundaries. 
The connected components of intersections $P_i
\cap P_j$ consist of horizontal and vertical segments, which we denote
respectively by $\xi_1, \ldots, \xi_{\ell_1}$ and $\eta_1, \ldots,
\eta_{\ell_2}$. We do not allow the singularities of $M$ to be in
the interior of a segment (if it is we subdivide the segment in
two). Each segment is attached to others at its two  
endpoints. If an endpoint of one of the $\eta_i$'s is a regular point of
$M$, attached to it are two of the $\eta_i$'s and at most two of the
$\xi_j$'s, and if it is a singularity more
segments will be attached. Orienting the $\xi$-edges from left to
right and the $\eta$-edges from top to bottom, we obtain a directed
graph with two kinds of edges, embedded in $M$ and equipped with
two additional structures:
\begin{itemize}
\item
At each vertex there is a cyclic order for
the incident edges, obtained 
by going around a small neighborhood of the point in $M$ in the
counterclockwise direction. Moreover along a cycle, consecutive edges
of the same kind have opposite orientations (i.e. if $\eta_i$ is
incoming and $\eta_j$ is the next $\eta$-edge then it is outgoing). 
\item
For each $\xi$ (resp. $\eta$) segment there are top and bottom
(resp. right and left) labels from $\{1, \ldots, m\}$ indicating the
parallelogram glued to the segment on the appropriate side. 
\end{itemize}
We call a graph with these structures a {\em gluing pattern} for a decomposition into
parallelograms. The gluing pattern obeys certain obvious
restrictions. For example, at a vertex $v$, if there is no $\xi$
edge between an incoming and an outgoing $\eta$ edge then the `right'
labels of the two $\eta$ edges are the 
same. Note that the gluing pattern determines the stratum
containing $M$. Since each vertex of a gluing pattern is either a
corner of a rectangle or a singularity of $M$, there are only finitely
many gluing patterns when the number of singularities and the number
of parallelograms are fixed. 

Also associated with a parallelogram decomposition is the {\em metric data}
consisting 
of the lengths of sides of the $P_i$'s and the lengths  of the
$\xi$- and $\eta$-edges. Given the gluing pattern and the lengths of
the edges, it is possible to calculate the sidelengths of the
$P_i$'s; thus the gluing pattern places some
restrictions on the metric information. It is clear that the gluing
pattern and the metric data taken together uniquely determine the
translation surface $M$. 

One could similarly define a gluing pattern for a parallelogram
decomposition of a half-translation surface. This will not be used in
the present paper and is left to the avid reader. 

\subsection{Fuchsian groups}\name{subsection: cones}
A Fuchsian group is a discrete subgroup of $G$. Fixing a haar measure
$\mu$ on $G$, we define the {\em covolume} of $\Gamma$ in $G$ as the measure
of a fundamental 
domain for the action of $\Gamma$ on $G$, and denote the covolume by
$\bar{\mu}(\Gamma)$. It is easily checked that if $M$
and $M'$ are affinely 
equivalent then $\bar{\mu}(\Gamma_M) = \bar{\mu}(\Gamma_{M'}).$ If
$\bar{\mu}(\Gamma_M) < \infty$ 
then $M$ is called a {\em lattice surface}.

An element of a Fuchsian group is called {\em parabolic} if it is
conjugate to $h_1$, and an automorphism $\varphi \in \Aff(M)$ is
called {\em parabolic} if 
$D\varphi \in \Gamma_M$ is parabolic.
The following is
well-known (see \cite{Veech - alternative}):
\begin{prop}
\name{prop: simple automorphisms}
If $\varphi \in \Aff(M)$ is parabolic and $D\varphi$ fixes the
direction $\theta$, then $M \sm \LL_M(\theta)$ is a cylinder
decomposition. For any $m$ there is $C = C_m$ such that if $m$ is the
number of cylinders in the above decomposition, then there is $k
\leq C_m$ such that $\psi = \varphi^k$ preserves the cylinders and the
corresponding inverse moduli $\mu_1, \ldots, \mu_m$ for $\psi$ satisfy 
$\LCM(\mu_1, \ldots, \mu_m) = \mu,$ where 
$
D\psi = r_\theta h_{\mu} r_{-\theta}$ and $n_i = \mu/\mu_i$ is the
number of Dehn twists which $\psi$ induces around a waist curve in
the $i$th cylinder. 
\end{prop}



For a Fuchsian group $\Gamma$ we now define cusp
areas and cone areas. An infinite cyclic subgroup of $G$
generated by a parabolic element is called {\em parabolic}.
Suppose that $\Gamma$ contains a parabolic subgroup $P$ and is {\em 
non-elementary}, that is,  
not a finite extension of an abelian group. Suppose also that $P$ is
{\em maximal}, i.e.\ not properly 
contained in a parabolic subgroup of $\Gamma$. Choose an element $g
\in G$ such that 
\eq{eq: prop of g}{
gPg^{-1} = \left \langle h_1  \right \rangle,
}
 and relabelling, replace $P$ and $\Gamma$ by
$gPg^{-1}$ and $g\Gamma g^{-1}$ respectively. Let $\HH$ be the complex upper half
plane, let $\ii = \sqrt{-1}$, let $\mathcal{C}_t$ be the image of the disk $\left\{z \in \HH:
\left|z-\frac{t}{2}\ii \right | < \frac{t}{2} \right\}$ in $\HH/P$, let
$\varphi: \HH/P \to 
\HH/\Gamma$ be the natural map, and let 
\eq{eq: defn t0}{
t_0=t_0(\Gamma, P) = \sup \{t>0 : \varphi|_{\mathcal{C}_t} \mathrm{\ is \ injective }
\}.
}
We will show in Propositions
\ref{prop: cusp}, \ref{prop: conjugation} that the set in the right
hand side of \equ{eq: defn 
t0} is nonempty and bounded above, so that $t_0$ is well-defined, and
that $t_0$ does not
depend on the choice of $g$ in \equ{eq: prop of g} and satisfies $t_0(\Gamma,
P) = t_0(x\Gamma x^{-1}, xPx^{-1})$ for $x \in G$. We call $t_0$ the
{\em cusp area} of $P$ in $\Gamma$; a simple computation shows that it
is equal to the hyperbolic area of $\mathcal{C}_{t_0}$.


\subsection{Orientation double cover}
Suppose $M$ is a half-translation surface. Then
there is a topological branched cover $\pi: \til S \to S$ of degree 2
such that the pulled-back flat surface $\til M$ is a translation surface called 
the {\em orientation double cover} of $M$, 
see \cite{DH}.
We have:
\begin{prop}
\name{prop: double cover}
Let $\til M \to M$ be the orientation double cover of a
half-translation surface $M$. Then:
\begin{itemize}
\item
Any automorphism of $M$ lifts to an automorphism
of $\til M,$ with the same derivative, so that $\Gamma_M \subset
\Gamma_{\til M}$. If a parabolic
(resp. hyperbolic) affine automorphism $\varphi$ of $M$ has an
associated cylinder decomposition (resp. Markov partition) with $k$
cylinders (resp. parallelograms) then the lifted automorphism $\til
\varphi$ of $\til M$ has an associated cylinder decomposition
(resp. Markov partition) with at most $2k$ 
cylinders (resp. parallelograms).
\item
Taking orientation double covers is $G$-equivariant, i.e. for $g \in
G$, $g \til M$
is the orientation double cover of $gM$. 
\item
Given $\til M$ there are at most finitely
many $M$ such that $\til M$ is the orientation double cover of $M$. 
\ignore{
\item
$\Gamma_M$ is contained in $\Gamma_{\til M}$ as a subgroup of finite
index, where the index is bounded above by number depending only on the
stratum containing $M$.  \combarak{is this needed? If not, remove, and
shorten proof}.}
\end{itemize}
\end{prop}

\begin{proof}
A proof of the fact that any affine automorphism lifts can be found in
\cite{Rykken}. It 
implies that there is an inclusion $\Gamma_M \subset \Gamma_{\til
M}$. 
The statement about the number of cylinders and rectangles follows
immediately from the fact that the degree of the cover is $2$. The
second statement is immediate from the construction of the orientation
double cover. 
For the third statement see e.g. \cite{Vorobets}. 
\end{proof}

\subsection{Perron-Frobenius}
A matrix $A  \in \Mat_d(\R)$ with all entries non-negative is called
{\em non-negative}, and a non-negative matrix is called {\em
irreducible} if 
for some $k$, $A^k$ has all its entries strictly positive. A vector
$\vec{v} \in \R^d$ is called {\em positive} if all its entries are
strictly positive.

We will need the following classical result (see e.g. \cite{Gantmacher}).
\begin{prop}
\name{prop: Perron Frobenius}
Suppose $A
\in \Mat_d(\R)
$ is an irreducible non-negative matrix. Then:
\begin{itemize}
\item
There is a unique (up to scaling) positive eigenvector $v^+$ of $A$,
and the corresponding eigenspace is one-dimensional.
\item
Let $\lambda$ be the eigenvalue for which $Av^+ = \lambda
v^+$. Then for any other eigenvalue $\beta$ of $A$, $|\beta| < \lambda.$ 

\end{itemize}
\end{prop}

We denote the eigenvalue $\lambda$ of Proposition \ref{prop: Perron
Frobenius} by $\lambda(A)$.
We will be particularly interested in irreducible non-negative matrices with integer entries. 

\begin{prop}
\name{prop: PF finiteness}
Given $T>0$ and $d \in \N$, the set  
$$\{A \in \Mat_d(\Z) : A \text{ {\rm is  irreducible  non-negative,} }
\lambda(A) < T \}$$
is finite.

\end{prop}

\begin{proof}
See \cite[Proof of Theorem 6]{PP}. 
\end{proof}

\section{Some hyperbolic geometry}
This section contains some simple propositions in hyperbolic geometry
related to cusp areas and cone areas. We will use the notation introduced in \S
2.3.

\begin{prop}
\name{prop: cusp}
Let $\Gamma$ be a Fuchsian group containing the
maximal parabolic subgroup $P =
\left \langle 
h_1
\right \rangle.
$ Then 
$X = \{t >0 : \varphi|_{\mathcal{C}_t} \mathrm{\ is \ injective }
\}$ is nonempty. 
If $\Gamma$ is nonelementary then $t_0 = \sup X <
\infty, $ and there is $\gamma_0 \in \Gamma$ which is a conjugate of
$h_1$, such that for some $s \in \R$, 
$h_s^{-1} \gamma_0 h_s = \til h_{-t_0^2}.$
The hyperbolic area of $\mathcal{C}_t$ is $t_0$.
\end{prop}

\begin{proof}
We denote the right-action of $\Gamma$ on $\HH$ by
$$
z \cdot \gamma = \frac{az+c}{bz+d}\, , \ \ \  \ \mathrm{where} \ \gamma = \left( \begin{matrix} a &
b \\ c & d
\end{matrix} \right) \in \Gamma.
$$
Let  $\mathcal{B}_t = 
\left \{z \in \HH: \left|z-\frac{t}{2}
\ii\right|<\frac{t}{2}\right\}$ and let $H_t = \partial
\, \mathcal{B}_t.
$
 By our conventions $\mathcal{B}_t$ projects to $\mathcal{C}_t$
and $H_t$  is the
horocycle based at 0 through $t\ii,$ which projects to a
periodic horocycle on $\HH/P$.  Let $S \subset H_1$ be a compact
segment whose projection to $\HH/P$ contains a full period of the
horocycle. Since the action of 
$\Gamma$ on $\HH$ is properly discontinuous, we can write
$$\{\gamma_1, \ldots, \gamma_N \} = \{\gamma \in \Gamma: S \cap S
\cdot \gamma \neq \varnothing, \, 0 \cdot \gamma \neq 0 \}.$$

To show that $X \neq \varnothing, $ take $0<t \leq 1$ small enough so
that $\mathcal{B}_t$ does not 
intersect $\bigcup_{i=1}^N \mathcal{B}_t \cdot \gamma_i$; such $t$ exist
because the $\mathcal{B}_t \cdot \gamma_i$ are not based at 0. Suppose if
possible that for some $\gamma \in \Gamma \sm P,$ $\mathcal{B}_t 
\cap \mathcal{B}_t \cdot \gamma \neq \varnothing.$ In particular
$\mathcal{B}_1 \cap \mathcal{B}_1 \cdot \gamma
\neq \varnothing$. It follows from discreteness of $\Gamma$ and
maximality of $P$ that $P$ is equal to its own normalizer in $\Gamma$.
This implies that $p = 0 \cdot 
\gamma \neq 0$; then $H_1 \cap H_1 \cdot \gamma \neq \varnothing$ (see
Figure 1) and since
any element of $H_1$ may be mapped to $S$ by suitable powers of
$h_1$, there are $j, k \in \Z$ such that $S \cap S \cdot \til \gamma \neq
\varnothing, $ where $\til \gamma = h_1^j \gamma h_1^k \in \Gamma.$ Thus for some
$i \in \{1, \ldots, N\}$ we have $\til \gamma = \gamma_i$ and
$\mathcal{B}_t \cap \mathcal{B}_t h_1^{-j} \gamma_i h_1^{-k} \neq
\varnothing.$ Since $\mathcal{B}_t$ is invariant under $h_1$, this
contradicts the definition of $t$. 

\begin{figure}[htp] \name{figure: horocycles}
\input{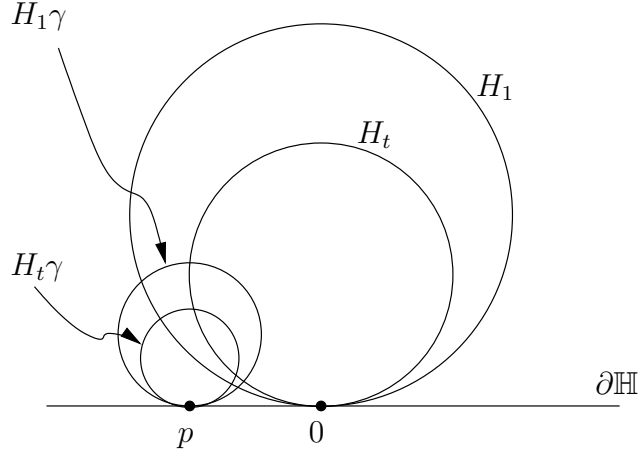}
\caption{Intersection of horocycles}
\end{figure}

Now if $\Gamma$ is nonelementary then it contains $\gamma$ so that
$0 \cdot \gamma \neq 0$, so that $H_1$ and $H_1 \cdot \gamma$ are horocycles
based at different points. Then for all sufficiently large $t$, $H_{t}
\cap H_{t} \cdot \gamma \neq \varnothing$, which shows
$\sup X < \infty.$ 
Further, by definition of $t_0$, there is $\gamma \in \Gamma$ such
that $H_{t_0} \cap H_{t_0} \cdot \gamma \neq
\varnothing$ and $\mathcal{B}_{t_0} \cap \mathcal{B}_{t_0} \cdot \gamma =
\varnothing,$ i.e., $H_{t_0}$ and $H_{t_0} \cdot \gamma$ are horocycles based
at different points and tangent to each other. We denote the point of
tangency by $p_0$. 

Since $p_0 \gamma^{-1} \in H_{t_0}$, there is $s_1$ such that 
$$p_0 = p_0 \cdot x, \ H_{t_0} \cdot x = H_{t_0} \cdot \gamma \ \ \
\mathrm{where} \ 
x=h_{s_1} \gamma.$$ 
We now express $x$ in terms of $t_0$, as follows. 

Let $s \in \R$ so that $\til p_0
= p_0 \cdot h_{s} = t_0 \ii$. The matrices 
\eq{eq: defn of a}{
a = \left( \begin{matrix} t_0^{-1/2} & 0 \\
0 & t_0^{1/2} \end{matrix} \right), \ \ \ \ \  
w = \left(\begin{matrix} 0 & -1 \\
1 & 0 \end{matrix} \right)
}
satisfy $\til p_0 \cdot a = \ii = \ii \cdot w,$  and $H_1 \cdot w = \til H$,
where $\til H = \{z \in \HH : \mathrm{Im} z = 1\} = \{\ii \cdot \til h_s : s \in
\R \}$ (see Figure 2). Thus the matrix
$x' = h_{s} a w a^{-1} h_{-s}$ also satisfies $p_0  = p_0 \cdot x', \,
H_{t_0} \cdot x' = H_{t_0} \cdot \gamma$, and this implies 
\eq{eq: form of x}{
x = \pm h_{s} a w a^{-1} h_{-s}.
}

\begin{figure}[htp] \name{figure: tangency}
\input{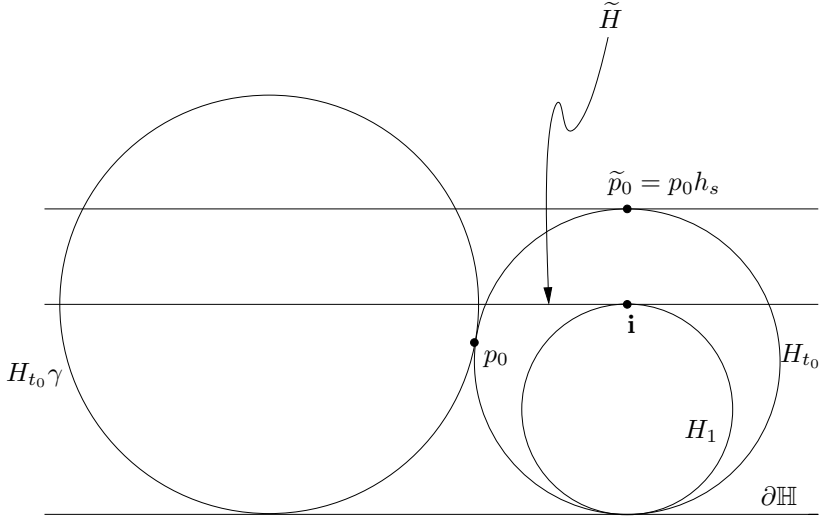}
\caption{Tangency of $H_{t_0}$ and $H_{t_0}\gamma$}
\end{figure}

Let $\gamma_0 =  x^{-1} h_1 x = \gamma^{-1} h_1 \gamma \in
\Gamma$. A computation using \equ{eq: form of x} gives  
\[
h_{s}^{-1} \gamma_0 h_{s}
=  \til h_{-t_0^2} .
\]

Finally, note that the hyperbolic area of $\mathcal{C}_{t_0}$ is the
same as that of $\mathcal{C}_{t_0} \cdot \gamma h_s$, which by the
above is covered bijectively by $\{z = x+\ii y \in \HH : y \geq t_0, \,
0 \leq x < t_0^2 \}$. Hence its area is
$$\int_0^{t_0^2} \int_{t_0}^{\infty} \frac{dy}{y^2} \, dx = 
t_0.$$
\end{proof}

\begin{prop}
\name{prop: conjugation}
The definition of $t_0(\Gamma, P)$ does not depend on the choice of
$g$ for which \equ{eq: prop of g} holds, and $t_0(x\Gamma x^{-1}, xPx^{-1}) =
t_0(\Gamma, P).$ 

\end{prop}

\begin{proof}
Let the notation be as in the the preceding proof. It was shown in the
proof that 
$$t_0 = \inf \{t >0: H_{t} \cap \bigcup_{\gamma \in \Gamma \sm P} H_t
\gamma \neq \varnothing\};$$
if we replace $g$ by $g'$ also satisfying  \equ{eq: prop of g}, then
$x=g'g^{-1}$ satisfies $xh_1x^{-1} = h_1$ and hence $x=h_s$ for some
$s$. This implies that $H_t x = H_t$ for all $t$ and hence $H_t \cap
H_t \gamma \neq \varnothing$ is equivalent to $H_t \cap H_t x \gamma
x^{-1} \neq \varnothing.$ This proves the first assertion. The second
assertion is an immediate consequence of the first. 
\end{proof}

Suppose $\Gamma$ is a Fuchsian group containing an elliptic element
$f$. The corresponding {\em cone} in $\HH/\Gamma$ is the image of
$B(z_f, R(f))$ under
the map $\varphi: \HH/\langle f \rangle \to \HH/\Gamma$, where $z_f
\in \HH$ is the fixed point of $f$ and 
$R(f)$ is defined via \equ{eq: defn R}. 

\begin{prop}
\name{prop: cones}
With the above notation, let $\gamma \in \Gamma$ such that $z_f \gamma$
is a closest point to $z_f$ in the orbit $z_f \Gamma$. Then the
commutator $h= [\gamma, f]$ is hyperbolic, its eigenvalue is bounded
above by a number depending only on $R(f)$, and the distance from the
axis of $h$ to $z_f$ is also bounded by a number
depending on $R(f)$.
\end{prop}

\begin{proof}
Write $z = z_f, z' = z \gamma$ and $f' = \gamma f
\gamma^{-1}$. Clearly  
$R(f) = d(z, z')/2.$
Conjugating $\Gamma$ with an appropriate element of $G$ and
relabelling we can arrange
that $z = t+\mathbf{i}$ and $z' = -t + \mathbf{i}$ are symmetric
with respect to the y-axis in $\HH$. Using the formula for hyperbolic
distance 
(see e.g. \cite[Thm. 1.2.6]{S.Katok}) we find $\cosh d(z,
z') = 1+ 2t^2$. 

On the other hand $f = h_t r_\theta h_{-t}$ and $f'
= h_{-t} r_\theta h_t$. Computing $h = f' f^{-1}$ directly
we find 
$\mathrm{tr}(h)=2 + 4t^2 \sin^2\theta.$
Since this number is greater than 2, $h$ is hyperbolic, and since this
number is no more than $2 \cosh d(z, z') = 2 \cosh 2R(f),$ and the
trace determines the eigenvalue, the second assertion follows. 

Since $f^{-1}$ fixes $z$ and $f'$ rotates around $z'$, the distance
which $z$ is moved by $h=f'f^{-1}$ is bounded in terms of $R(f)$. 
On the other hand the order of $f$ bounds $\sin \theta$ from below and
and hence implies a lower bound on the displacement of $h$. Now let
$p_1, p_2$ be the points on the axis of $h$ closest to $z$ and $zh$
respectively, so that $p_2 = p_1h$; consider the quadrilateral joining
$z, p_1, p_2, zh$. We have bounded $d(p_1, p_2)$ from
below and $d(z, zh)$ from above, and the angles at $p_1, p_2$ are
right angles. 
From this it is easy to deduce an upper bound on the
length $d(p_1, z)$. 
\end{proof}

\section{Finiteness of small cusps}
In this section we prove a more precise version of Theorem
\ref{thm: cusp areas finite}. The idea is that when $\Gamma_M$ is
non-elementary, a cusp in $\HH/\Gamma_M$ gives rise to two transverse
cylinder decompositions on $M$, one an image of the other under an
affine automorphism, and the
cusp area bounds the combinatorics of the corresponding intersection
pattern. We will first consider translation surfaces, so denote
$$\NLC_1(m,T) = \left\{(M,P) \in \NLC(m,T): M \mathrm{\ is \ a \
translation \ surface}\right \},$$
and by $\til \NLC_1(m,T)$ the corresponding sets of
$G$-orbits.

For any $T>0$ and $m \in \N$, let $\mathcal{N}(m,T)$ denote the set of
pairs $(A,D)$, where $A, D \in \Mat_{m}(\Z_+), \, A $ is symmetric,
$D$ is diagonal, $DA$ is non-negative irreducible and 
$\lambda(DA)<T.$ The following is an immediate corollary of
Proposition \ref{prop: PF finiteness}.

\begin{prop}
\name{prop: PF pairs}
For any $m \in \N$ and $T>0$, $\# \mathcal{N}(m,T) < \infty.$ 
\end{prop}

Given a symmetric $A = \left( a_{ij} \right) \in \Mat_{m}(\Z_+)$, let
$\ell =\sum_{i,j} a_{ij}$. 
Denote by
$\mathcal{P}(A)$ the set of simultaneous conjugacy classes
$[(\sigma_1, \sigma_2)]$ for which (i) and (ii) of \S2.2 hold, and such
that 
$\sigma_1$ and $\sigma_2$ are conjugate in $S_{\ell}.$

\ignore{
\begin{remark}
It would be interesting to know $\#
\mathcal{N}(m,T)$ and $\# \mathcal{P}(A)$ explicitly, or understand their asymptotics.
\end{remark}
}

\begin{thm}
\name{thm: NLC precise}
For fixed $T>0$ and $m \in \N$, 
$$\# \til \NLC_1(m, T) \leq
\sum_{(A, D) \in \mathcal{N}(m, C_m T)} \# \mathcal{P}(A).$$ 
\end{thm}

\begin{proof}
We will construct a map 
$$\Psi: \NLC_1(m, T) \to X,$$
where 
$$X=\left\{ \left (A,D, [(\sigma_1, \sigma_2)] \right): (A,D) \in
\mathcal{N}(m, C_mT), \, [(\sigma_1, \sigma_2)] \in \mathcal{P}(A) \right\},
$$
such that $\Psi (M,P) = \Psi(M', P')$ if and only if there is $g \in
G$ such that $M' = gM$ and $P' = gPg^{-1}$; in other words, $\Psi$
induces an injective map 
defined on $\til \NLC_1(m, T).$ 

Let $(M', P') \in \NLC_1(m, T),$ that is $M'$ is a flat surface,
$\Gamma_{M'}$ is non-elementary, and $P' \subset \Gamma_{M'}$ is maximal
parabolic. By a conjugation we may assume that $P'$ is the cyclic
group generated by
$h_1$. Let $s, \gamma_0, h_1, t_0$ be as in
Proposition \ref{prop: cusp}, let $a$ be as \equ{eq: defn of a}, let $g= a^{-1}
h_{-s}$, and let $M=gM'$. Then $\Gamma_M = g\Gamma_{M'}g^{-1}$ contains
the two elements 
\eq{eq: gammai}{
\gamma_1 = h_{t_0}, \ \ \ \ \ \ \ \gamma_2=\til h_{-t_0}, 
}
and there is $\gamma \in \Gamma_M$
such that 
\eq{eq: prop of gamma}{
\gamma_2 = \gamma^{-1} \gamma_1 \gamma.
}
Moreover $P = gP'g^{-1}$ is generated by $\gamma_1$. 

By Proposition \ref{prop: simple automorphisms},
there is a parabolic affine automorphism $\varphi_1 \in \Aff(M)$ and
$k \leq C_m$ such that $D \varphi_1 =
\gamma_1^{k}= h_{k t_0},$ and $\varphi_1$ preserves the cylinders in a
horizontal cylinder 
decomposition $M= C_1 \cup \cdots \cup C_{m}$. 
Let $\psi \in \Aff(M)$ such that $\gamma = D \psi$, and let $\varphi_2
= \psi^{-1} \varphi_1 \psi \in \Aff(M)$, so that 
$D \varphi_2 =\gamma_2^{k}.$
Let $C'_i = \psi^{-1}(C_i)$. Then $M= C'_1 \cup \cdots \cup C'_m$ is a
vertical cylinder decomposition invariant under $\varphi_2$. 

 Let $\vec{a}, \vec{w}$ 
(resp. $\vec{a}', \vec{w}'$) be positive vectors in $\R^m$ recording
the heights and circumferences of the cylinders $C_1, \ldots, C_m$
(resp. $C'_1, \ldots, C'_m$). It follows from \equ{eq: gammai} and \equ{eq: prop of gamma}
that $\gamma = h_{s_1} w \til h_{s_2}$ for some $s_1, s_2 \in \R$, where
$w$ is as in \equ{eq: defn of a}. Since the action of $h_s$
(resp. $\til h _s$) does not affect the heights and circumferences of
horizontal (resp. vertical) cylinders, and since $w$ affects a
rotation by $\pi/2$, we get
$$\vec{a} = \vec{a}', \ \ \ \vec{w}=\vec{w}'.
$$

Now let $A = (a_{ij}) \in \Mat_{m} (\Z)$ where $a_{ij}$ is
the number of connected components of $C_i \cap C'_j$. By
connectedness of $M$, $A$ is irreducible. Since $\psi$
exchanges the roles of $C_i$ and $C'_j$, $A$ is symmetric. By
construction, the gluing pattern $[(\sigma_1, \sigma_2)]$
corresponding to the two cylinder decompositions above is in
$\mathcal{P}(A).$ 
We have
\eq{eq: w}{
\vec{w}= \vec{w}' = A \vec{a}.
}

Now let $n_i$ be the number of Dehn twists induced by $\varphi_1$
around a waist curve in $C_i$. By Proposition \ref{prop: simple automorphisms} we have 
$ k t_0 = n_i w_i/a_i,$
that is, if $D = \diag(n_1, \ldots, n_m) \in \Mat_{m}(\N)$,
then
$$k t_0 \vec{a} = D \vec{w} = D A \vec{a}.$$
That is, $k t_0 = \lambda(DA)$, $(D,A) \in \mathcal{N}(k T)$ and $k
\leq C_m.$
Altogether we have $(D,A, [(\sigma_1, \sigma_2)]) = \Psi(M', P) \in X,$
so to prove the theorem it remains to show that the $G$-orbit of $(M',
P')$
may be reconstructed from $(D, A, [(\sigma_1, \sigma_2)]).$ 
Note that $\vec{a}$ is uniquely determined up to
scaling as the
positive eigenvector of $DA$, and $\vec{w}$ is determined from
$\vec{a}$ and $A$ via \equ{eq: w}. The scaling parameter is also
uniquely determined by the requirement that $M$ has area one. So
$(A,D)$ determine $\vec{a}$ and $\vec{w}$, that is, the heights and
circumferences of the rectangles, and $[(\sigma_1, \sigma_2)]$
determines how they are to be glued to each other. This determines
$M$, and $P$ is the maximal parabolic subgroup of $\Gamma_M$ leaving
the horizontal direction fixed.
\end{proof}

\begin{proof}[Proof of Theorem \ref{thm: cusp areas finite}.]
By Proposition \ref{prop: double cover}, for each $(M,P) \in
\NLC(m,T)$, the orientation double cover $\til M$ also has $P$ as a
parabolic subgroup of $\Gamma_{\til M}$, and the corresponding
cylinder decomposition has at most $2m$ cylinders. Thus, if $\til P
\subset \Gamma_{\til M}$ is a maximal parabolic subgroup containing
$P$, then $(M,P)$ gives rise to $(\til M, \til P) \in \NLC_1(2m,
T)$. Moreover for a fixed 
$(\til M, \til P)$ there are only finitely many $(M, P)$ from which it
arises in this way. The Theorem follows. 
\end{proof}

\begin{proof}[Proof of Corollary \ref{cor: covolumes finite}.]
In a given stratum there is an upper bound on the number of cylinders
in a cylinder decomposition. See \cite[Lemma, pg. 302]{KMS} or
Smillie's improved bound as presented in \cite{yoav}. 
Suppose by contradiction that $M_1, M_2, \ldots$ are infinitely
many affinely inequivalent surfaces, such that for each $i$,
$\HH/\Gamma_i$ does not contain an embedded ball of radius $R$. Here
$\Gamma_i = \Gamma_{M_i}$. By
Theorem \ref{cor: new} the $M_i$ are lattice 
surfaces, hence for each $i$, $\Gamma_i$ contains at least one
maximal parabolic $P_i$. By Theorem \ref{thm: cusp areas finite},
$t_0(\Gamma_i, P_i) \to \infty.$ A cusp with cusp area $t_0$ contains an
embedded ball of radius $R(t_0)$, where $R(t_0) \to \infty$ as $t_0
\to \infty$ --  a contradiction proving (i). 
For a
lattice surface, the cusp area is bounded 
above by the covolume, implying (ii). 
\end{proof}

\begin{remark}
It is also possible to deduce Corollary \ref{cor: covolumes finite}
 and the finiteness of $\til
\SC(m, T) \cap \MM$ for any stratum $\MM$, from 
Proposition \ref{prop: Thurston bound}, as follows. To any cusp one
associates the element $h=\gamma_1 \gamma_2$ as in \equ{eq: gammai} and
proves it is hyperbolic, with $\lambda(h)$ bounded by a number
depending on the corresponding cusp area. Our Theorem \ref{thm: cusp
areas finite} is
stronger in that it does not assume a bound on the topology of the
surface, but only on the number of cylinders in the corresponding
cylinder decomposition. 

\end{remark}

\section{Hyperbolic affine automorphisms and Markov partitions}
In this section we prove Proposition
\ref{prop: Thurston bound}. As noted  in the introduction, the
result is not new and presumably our argument is also
well-known. We include it since it does not appear in the literature. 

Let $M$ be a flat surface and let $\varphi$ be a hyperbolic affine
automorphism (in a different terminology, $\varphi$ represents a {\em pseudo-Anosov
homeomorphism}). 
Let $h = D\varphi$ be the corresponding element of $\Gamma_M$, and
let $\lambda=\lambda(h)$. By applying $g \in G$, let us assume that $h$
expands the x-axis and contracts the y-axis by a factor of
$\lambda$. A
{\em Markov partition} for $\varphi$ is a covering $M = P_1 \cup
\cdots \cup P_p$ by closed rectangles, with disjoint interiors,
horizontal and vertical sides, such that the following {\em Markov
property} is satisfied: for each $i$,
the image of any vertical (resp. horizontal) side of any $P_i$ under
$\varphi$ (resp. $\varphi^{-1}$) is contained in a vertical
(resp. horizontal) side of one of the $P_j$. This means that if
$\interior \, \varphi(P_i)$ intersects the interior of some $P_j$ then
it extends all the 
way through to both sides. 
It is known that a Markov partition exists. A sketch of proof is given
in \cite{travaux}, and we have provided more details in an appendix.

The {\em intersection matrix of $\varphi$} is the $p \times p$ integer
matrix $A = (a_{ij})$, where $a_{ij}$
is the number of connected components of $\interior \, \varphi(P_i) \cap
\interior \, P_j$. One has:

\begin{prop}
\name{prop: Markov, A}
\begin{itemize}
\item
$A$ is irreducible.
\item
$\lambda = \lambda(A)$. 
\item 
The vector recording the widths (resp. heights) of the rectangle
$P_i$ is a positive eigenvector for $A$ (resp. $A^{\mathrm{tr}}$). 
\end{itemize}
\end{prop}

\begin{proof}
It follows from the Markov property that the $i,j$th entry of $A^n$
counts the number of components in the 
intersection $\varphi^{n}(P_i) \cap P_j$. Suppose by contradiction
that $A$ is not irreducible, so that there is 
a sequence $n_k \to \infty$ and $1 \leq i,j \leq p$ so that
\eq{eq: to contradict1}{
\varphi^{n_k}(P_i) \cap P_j = \varnothing.
}
Let $\sigma$ be a
horizontal segment in $P_i$ and consider its image $\sigma_k$ under
$\varphi^{n_k}$. Then $\sigma_k$ is a horizontal segment contained in
$\varphi^{n_k} (P_i)$ whose length tends to infinity. Passing to a
subsequence we can 
assume the left endpoints of $\sigma_k$ converge to $x \in M$, so that
any point along an infinite horizontal ray $\ell$ issuing from $x$ is a limit
of points of $\sigma_k$. 
Since $\varphi^{-1}$ is an affine automorphism on $M$
which contracts horizontal saddle segments, there are no horizontal
saddle connections on $M$, so the horizontal line flow on $M$ is
minimal. This implies that $\ell$ is dense in $M$. In particular for 
large enough $k$, $\sigma_k \cap \interior \, P_j \neq \varnothing$,
contradicting \equ{eq: to contradict1}.

Now let $w_i$ be the width of $P_i$, so that $\lambda w_i$ is the
width of $\varphi(P_i).$ By the Markov property $\varphi(P_i)$ passes $a_{ij}$ times
through $P_j$, so that 
$\lambda w_i = \sum_j a_{ij} w_j$. That is, 
$$\lambda \vec{w} = A \vec{w}.$$
Since $\vec{w}$ is a positive vector, $\lambda = \lambda(A)$. 
The proof for heights is almost identical. 
\end{proof}

For a non-negative irreducible $A \in \Mat_{p}(\Z)$, let
$\mathcal{G}(A)$ denote the gluing patterns (as in \S2.2.2) of Markov partitions 
arising from hyperbolic affine automorphisms whose intersection matrix
is $A$. We have:

\begin{prop}
\name{prop: bound on number of gluing patterns}
$\mathcal{G}(A)$ is finite. 

\end{prop}

\begin{proof}
Since $\varphi$ (resp. $\varphi^{-1}$) preserves the boundaries of
rectangles, but contracts 
vertical (horizontal) saddle connections, there are there are no
vertical or horizontal saddle connections for $M$. Thus each rectangle
contributes at most $4\pi$ to the total angle around the singularities
of $M$. In other words the number of rectangles bounds the number of
singularities. Since each edge in a gluing pattern is either bounded
by a singularity or by a full size of a rectangle, this also bounds
the number of edges in a gluing pattern. Thus the number of gluing
patterns is finite. 
\end{proof}

We now want to show that intersection matrix of $\varphi$ determines
the metric data for the gluing pattern. 
Let $\xi_i$ be the horizontal segments of the gluing
pattern corresponding to $P_1, \ldots, P_p$. We will
refine our Markov partition to obtain another partition with the Markov property. Each
$\xi_i$ is a connected component of the intersection of two of the
$P_j$, say $P_1, P_2$. Consider the vertical segments 
(if any) which issue from the 
endpoints of $\xi_i$ into the interior of $P_1, P_2$. Any such
segment divides a rectangle vertically into two rectangles, with the same height
as before. At least one of the new rectangles will have width
$\xi_i$. We denote the resulting rectangle decomposition by 
$M = Q_1 \cup \cdots \cup Q_r$, and claim that this decomposition also has the Markov
property. To see this, note first that subdividing rectangles by vertical lines
does not affect the requirement for $\varphi^{-1}$, which involves
only horizontal sides. Now suppose $Q \subset P_1, P_2$ are as in
Figure 3,
$Q'
\subset P'$ such that $\interior \, \varphi(Q) \cap \interior \, Q'
\neq \varnothing$.  There are two cases to consider:
\begin{enumerate}
\item If $\varphi(P_1)$
does not reach the top of $P'$ then also $\interior \, \varphi(P_2)
\cap \interior \, P' \neq \varnothing$ and by the Markov property for
the $P_i$, both $\varphi(P_1)$ and $\varphi(P_2)$ go across $P'$ so
$\varphi(Q)$ goes across $P'$ and hence across $Q'$. 
\item
If $\varphi(P_1)$ reaches the top of
$P'$ and $P''$ lies above $P'$ then $\varphi(P_1)$ goes across $P'$
and $\varphi(P_2)$ goes across $P''$ so that $\varphi(Q)$ goes across
the smaller of the two, i.e., $\varphi(Q)$ goes across $Q'$.
\end{enumerate}

\begin{figure}[htp] \name{figure: Markov}
\input{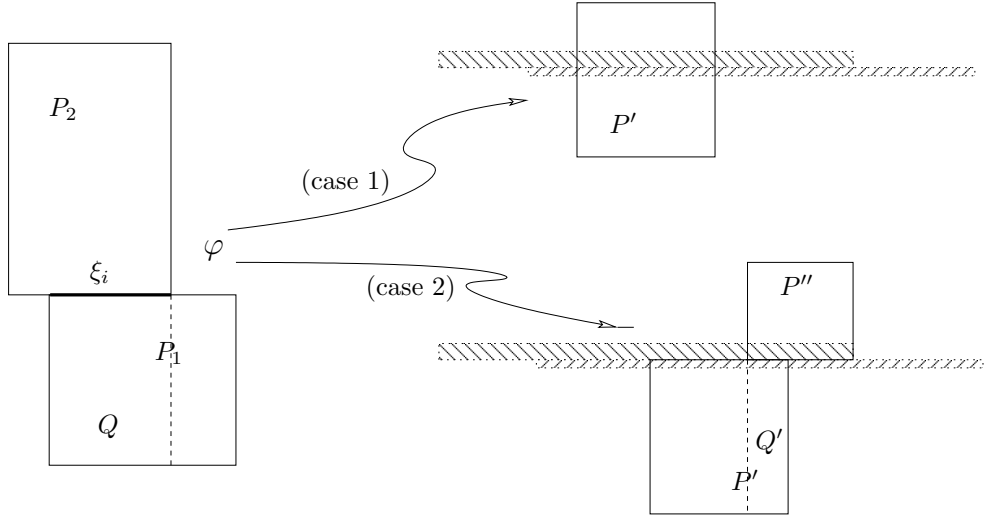}
\caption{Markov property for the $Q_j$}
\end{figure}

Considering all cases in this manner proves the claim. 

We can also refine using the $\eta_i$ instead of the $\xi_i$, cutting
rectangles horizontally, to obtain another Markov partition $T_1,
\ldots, T_t$. We define matrices $B$ and $C$ by letting $b_{ij}$
(resp. $c_{ij}$) be the number of connected components of
$\interior \, \varphi(Q_i) \cap \interior \, Q_j$ (resp. $\interior \,
\varphi(T_i) \cap \interior \, T_j$).
Repeating the proof of Proposition \ref{prop: Markov, A} we obtain:
\begin{prop}
\name{prop: Markov, B}
\begin{itemize}
\item
$B$ and $C$ are irreducible.
\item
$\lambda = \lambda(B) = \lambda(C)$. 
\item 
The vector recording the lengths of the segments $\xi_i$
(resp. $\eta_i$) is a positive eigenvector for $B$ (resp. $C$). 
\end{itemize}
\end{prop} 

We are now in a position to formulate and prove a more precise version
of Proposition \ref{prop: Thurston bound}. We will consider separately
translation and half-translation surfaces, so we write
$$\NSMP_1(p,T) = \{(M, h) \in \NSMP(p,T): M \mathrm{ \ is \ a \
translation \ surface} \},$$
and denote the corresponding set of $G$-orbits by $\til \NSMP_1(p,T).$
Let $\mathcal{M}(p,T)$ be the set of positive integer matrices $A$
with $\lambda(A)<T$. This is a finite set by
Proposition \ref{prop: PF finiteness}. Then we have:

\begin{thm}\name{thm: MP with bound}
For a fixed $T>0$ and $p \in \N$, 
$$\# \, \til \NSMP_1(p,T) \leq \sum_{A \in \mathcal{M}(p, T)} \# \, \mathcal{G}(A).$$
\end{thm}

\begin{proof}
Given $(M,h) \in \NSMP_1(p,T)$, we have constructed $A \in
\mathcal{M}(p,T)$ and a gluing pattern in $\mathcal{G}(A)$, and these
data only depend on the affine equivalence class of $M$. Thus we have
defined a map 
$$\Phi: \til \NSMP_1(p, T) \to \cup_{A \in \mathcal{M}(p,T)} \mathcal{G}(A),
$$
and it remains to show that $\Phi$ is injective, i.e. that $M$ and $h$
are uniquely determined by the matrix $A$ and the gluing pattern, up
to an element of $G$.   
Applying an element $g \in G$ we can
assume that the parallelograms in the Markov partition corresponding
to $(M, h)$ are rectangles with the expanding
direction horizontal, and moreover the width of the widest
rectangle can be normalized to be 1. With this choice of $g$ we need to
show that the matrix and gluing pattern uniquely 
determine $gM$. Since the widths of the rectangles are in the unique
positive eigendirection of $A$, and by our scaling convention, the matrix $A$
determines the widths of the rectangles. The heights also span the
unique eigendirection for $A^{\mathrm{tr}}$ so are determined by $A$
up to scaling. By the requirement that $gM$ has area 1, the heights
are uniquely determined by $A$. From $A$ and the gluing pattern one
determines the matrices $B$ and $C$, and the lengths of the $\xi_i$
and $\eta_j$ give a positive eigenvector of $B$ and $C$ which is thus
determined up to scaling. Since the sidelengths of the rectangles
$P_i$ are determined, and can also be calculated using the lengths of
the $\xi_i$ and $\eta_i$, the $\xi_i$ and $\eta_j$ are also uniquely
determined. Since we have specified its gluing pattern and
geometry, the flat surface $gM$ is uniquely determined.
\end{proof}

\begin{proof}[Proof of Proposition \ref{prop: Thurston bound}]
By Proposition \ref{prop: double cover}, for each $(M, h) \in \NSMP(p,T)$
the orientation double cover $\til M$ has a hyperbolic affine
automorphism $\varphi$ such that the corresponding number of
parallelograms is at most $2p$. I.e., $(M,h) \in \NSMP(p,T)$ gives
rise to $(\til M, h) \in \NSMP_1(2p, T)$. Also for a fixed
$(\til M, h)$ the number of $(M, h)$ covered in this way is
finite. The Proposition follows. 
\end{proof}


\section{Restrictions on Veech groups}
\begin{proof}[
Proof of \ref{cor: restriction on Gamma}]
It is easily checked that (I) and (II) depend only on the
commensurability class of $\Gamma$, so to prove them we can assume
that $\Gamma = 
\Gamma_M$ for a flat surface $M$. 
Since the maximal number of cylinders for a cylinder decomposition on
$M$ is bounded, assertion (I) follows immediately from Theorem
\ref{thm: cusp areas finite}. Similarly assertion (II) is immediate
from Proposition \ref{prop: Thurston bound}. 

To prove assertion (III) we note that in a Veech group there is a uniform
upper bound on the order of an elliptic element by Hurwitz's theorem,
hence such a bound also exists in any group commensurable with a Veech
group. So it remains to show that a Fuchsian group
$\Gamma$ in which (II) holds and orders of elliptic elements are
uniformly bounded, (III) holds. To see this, given $T$ suppose we have
a list $z_1, z_2, \ldots$ of fixed points
for elliptics $f_1, f_2, \ldots$ in $\Gamma$, such that the corresponding cone
areas are no more than $T$. Since the order of the $f_i$ is bounded above, we have
a uniform upper bound for $R(f_i)$ defined via \equ{eq: defn R}. Let
$h_i$ be the hyperbolic element corresponding to $f_i$ via Proposition
\ref{prop: cones}. Since the eigenvalue of the $h_i$ is bounded, by
(II) the list $h_1, \ldots, $ contains only finitely many conjugacy
classes. Conjugating the $f_i$'s and correspondingly the $h_i$'s we
may assume that there are only finitely many distinct elements in the
list $h_1, h_2, \ldots$. Given $h=h_i$ in this list, let $A$ be its
axis, and let $A_0 \subset A$ be a compact fundamental domain for the
action of $h$ on $A$. By 
a further conjugation of $f_i$ assume that the closest point to $z_i$
on $A$ is in $A_0$. In view of the upper bound on the order of the
$f_i$'s, by Proposition \ref{prop: cones} $z_i$ is within bounded
distance of $A_0$. By discreteness of $\Gamma$ there
are only finitely many $z_i$'s and hence finitely many $f_i$'s. 
\end{proof}

\begin{remark}
The argument proving (III) can be adapted to prove the following
analogue of Theorem \ref{thm: cusp areas finite} and Proposition
\ref{prop: Thurston bound}: for any stratum $\MM$ and any $T$, the set
of affine equivalence classes of pairs $(M, f)$, where $M \in \MM$ and
$f \in \Gamma_M$ is elliptic, with corresponding cone area at most
$T$, is finite. 

\end{remark}

\ignore{
\begin{proof}[Proof of Corollary \ref{cor: markov partitions in
stratum}]
It follows from the construction of Markov partitions (see Proposition
\ref{prop: Markov partitions exist}) that for any stratum 
$\MM$ there is $m_0$ such that any $\varphi$ on a 
surface $M \in \MM$ has a Markov partition for which the number of
parallelograms is at most $m_0$.
Thus the
finiteness of $\til \Hyp (\MM, T)$ follows from Proposition \ref{prop:
Thurston bound}. 
\end{proof}
}
For the proof of Theorem \ref{cor: normalizer} we will need the
following Lemma. 

\begin{lem}
\name{lem: normal subgroup}
Suppose $\Lambda_0$ is a Fuchsian group, $\Gamma$ a normal subgroup of
$\Lambda_0$,
$\gamma \in \Gamma$ a non-central element such that for any $\lambda
\in \Lambda_0$ there is $\tau \in \Gamma$ with 
$\gamma^{\lambda} = \gamma^{\tau}.$ 
Then $\Gamma$ is of finite index in $\Lambda_0$. 
\end{lem}

\begin{proof}[Proof of Theorem \ref{cor: normalizer} (assuming Lemma
\ref{lem: normal subgroup})]
Let $\Gamma$ be a non-elementary Fuchsian group commensurable
to a Veech group, let $\Lambda$ be the normalizer of $\Gamma$ in $G$,
and suppose 
$\Lambda/\Gamma$ is infinite. Since $\Gamma$ is non-elementary it
contains a hyperbolic element $\gamma$ \cite[Thm. 2.4.4]{S.Katok}, and
$\Lambda$ is Fuchsian \cite[Thm. 2.3.8]{S.Katok}. 
Let 
$$\Lambda_0 = \{\lambda \in \Lambda: \exists \tau \in \Gamma \mathrm{\
s.t. \ } \gamma^{\lambda} = \gamma^{\tau}\}.$$ 
In the action of $\Lambda$ 
on conjugacy classes of $\Gamma$, $\Lambda_0$ is the stabilizer of the
conjugacy class of $\gamma$. In particular $\Lambda_0$ is a subgroup
of $\Lambda.$ 
According to Lemma \ref{lem: normal subgroup},
$\Lambda_0/\Gamma$ is finite so $\Lambda/\Lambda_0$ is
infinite. This implies that the conjugates
$\{\gamma^{\lambda} : \lambda \in \Lambda\}$ comprise infinitely many
conjugacy classes, in other words $\Gamma$ contains infinitely many
hyperbolic elements which are conjugate 
in $\Lambda$, hence have the 
same eigenvalue, but are not conjugate in $\Gamma$. This 
contradicts Corollary \ref{cor: restriction on Gamma}. 
\end{proof}

\begin{proof}[Proof of Lemma \ref{lem: normal subgroup}]
Let $C$ be the centralizer of $\gamma$ in $\Lambda_0$ and let $C_0 = C
\cap \Gamma$. Since $\Lambda_0$ is Fuchsian, $C$ is cyclic (see
\cite[\S2.3]{S.Katok}) and therefore $C_0$ is of finite index
in $C$. By assumption for
every $\lambda \in \Lambda_0$ there is $\tau \in 
\Gamma$ such that $\lambda^{-1} \tau \in C$, which implies that
$\Lambda_0 =C \Gamma$. 
Thus $C$ maps onto $\Lambda_0 / \Gamma$. Since this
surjection factors through $C/C_0$, $\Lambda_0/\Gamma$ must be
finite. 
\end{proof}

\begin{proof}[Proof of Corollary \ref{cor: group determines surface}]
Given $\MM$, for any $M \in \MM$ and any hyperbolic $h \in \Gamma_M$,
the minimal number $p(M,h)$ of parallelograms in 
a corresponding Markov partition can assume only finitely many values
by Proposition \ref{prop: Markov partitions exist}. Let
$\Gamma$ contain a hyperbolic element $h$. 
By Proposition \ref{prop:
Thurston bound}
the number of $G$-orbits of $M$ such that $h\in \Gamma_M$ is
finite. In particular the set in \equ{eq: M(Gamma)} intersects only
finitely many $G$-orbits. If $M_1$ and $M_2$ are flat
surfaces in $\MM$ such that $gM_1 = M_2$ and $\Gamma_{M_1}  =
\Gamma_{M_2}$ then $g \in N$; and if $\Gamma$ is non-elementary then
$N/\Gamma$ is finite by Corollary \ref{cor: normalizer}. 
\end{proof}

\begin{proof}[Proof of Theorem \ref{cor: new}]
Given $R$ let $T = 2R$. Enlarging $T$ we ensure that a cusp of area at
least $T$ contains an embedded ball of radius $R$. Enlarging $T$
further, in light of the upper bound on the order of elliptic elements
in $\Gamma$, we ensure that any cone of area at least $T$ contains an
embedded ball of radius $R$. Thus we may assume $\HH/\Gamma$ contain
neither a  
cusp of volume at least $T$, nor a cone of area at least $T$; in view
of Corollary \ref{cor: restriction on Gamma}, $\HH/\Gamma$ has only
finitely many cusps and cones, and only finitely many
closed geodesics of length less than $T$. 

Now let $\mathcal{N}_1 \subset \HH/\Gamma$ be the union of closed
geodesics of length less 
than $T$, cusps of area less than $T$, and cones of area less than
$T$. Then $\mathcal{N}_1$ has finite area. 
Let $\mathcal{N}$ be a neighborhood of $\mathcal{N}_1$ which is
large enough so that a closed loop intersecting the complement of $\mathcal{N}$
and homotopic to either a geodesic of length less than $T$, a loop
around a cusp of area less than $T$, or a loop around a cone point of
area less than $T$, must have length at least $2R$. Since $\HH/\Gamma$
has only finitely many cusps and cones, the area of $\mathcal{N}$ is
also finite, so $\HH/\Gamma \neq
\mathcal{N}$. 

Let $\pi: \HH \to \HH/\Gamma$ be the natural map, and let $\bar{z} \in
\HH$ so that $z = \pi(\bar{z}) \notin \mathcal{N}$. We claim that
$\pi|_{B(\bar{z}, 
R)}$ is injective. Otherwise there is a segment $\bar{\sigma}$
connecting two distinct points in $B(\bar{z}, R)$, mapping to
a closed loop $\sigma$ in $\HH/\Gamma$ of length less than
$2R$. Either $\sigma$ has a shortest representative which is a
geodesic of length less than $2R$, or $\sigma$ is freely homotopic to
a curve around a cusp or cone, that is a curve in $\mathcal{N}_1$. But
then the definition of $\mathcal{N}$ ensures that the length of
$\sigma$ is greater than $T$, and this is a contradiction. 
\end{proof}

\begin{appendix}
\section{The construction of Markov partitions}
\begin{prop}
\name{prop: Markov partitions exist}
For any stratum $\MM$ of flat surfaces, 
there are $R_1$ and $R_2$ such that 
for any surface $M \in 
\MM$ and any hyperbolic $\varphi \in \Aff(M)$:
\begin{itemize}
\item
 any Markov partition
for $\varphi$ has at least $R_1$ parallelograms; 
\item
there is a Markov
partition for $\varphi$ with at most $R_2$ parallelograms.  
\end{itemize}

\end{prop}

\begin{proof}
Let $k = \sum_{\sigma \in \Sigma} k_\sigma$ so that $k\pi$ is the
total angle around 
all the singularties of a surface in $\MM$. Suppose $M \in \MM$ and
$\varphi \in \Aff(M)$ is hyperbolic, with a Markov partition into
parallelograms $P_1, \ldots, P_p$. By conjugating we can assume the
$P_i$ are rectangles with horizontal and vertical sides. The
horizontal and vertical directions are contracted by either $\varphi$
or $\varphi^{-1}$, but these maps 
preserve the set of saddle connections, so there
are no horizontal or vertical saddle connections on $M$. In 
particular each
edge of each $P_i$ can only contain one singularity, so that the total
angle around singularities coming from each $P_i$ is at most
$4\pi$. This implies that $p \geq R_1=k/4$, proving the first assertion.  

To prove the second assertion, by Proposition \ref{prop: double cover}
it suffices to consider strata 
of translation surfaces. Also, if we have a Markov partition for
$\varphi^r$, the common refinement of its images under $\varphi^i$,
$i=0, \ldots, r$, 
provides a Markov partition for $\varphi$. Thus by passing to a
suitable finite power of $\varphi$, which can be taken to depend only
on $\MM$, we can assume that $\varphi$ fixes all the singularities and
all the critical leaves issuing from singularities. 
Assume
without loss of generality that the contracting direction for
$\varphi$ (resp. $\varphi^{-1}$) is vertical (resp. horizontal). 

Since there are
no horizontal or vertical saddle connections, the foliation in both
the horizontal and 
vertical direction is minimal, so the leaves intersect any transverse
segment. 
We will construct the edges of the partition in 3 steps. 
\begin{enumerate}
\item
Let $\gamma$ be a short closed segment starting at a singularity
and going along a horizontal leaf, say from left to right, and
let $\lambda$ be a union of closed segments starting at each singularity
along the vertical leaves (going both up and down) until the first
point of intersection with $\gamma$. Denote the segments comprising
$\lambda$ by $\lambda_1, \ldots, \lambda_k$.
\item
The intersection $\gamma \cap \lambda$ is finite, 
so there is a {\em terminal point} $p$ which is the right-most
intersection point 
of $\lambda \cap \gamma$. Remove
from $\gamma$ the subsegment to the right of $p$. We retain the name
$\gamma$ for the shorter segment.

\item
Suppose $\lambda_j$ ends at $p$. Continue it further until its next
intersection with $\gamma$. We retain the names $\lambda, \lambda_j$ for
this new collection. Since $M$ has no vertical saddle
connections, the new intersection point of $\lambda_j$ with $\gamma$
is not one of the previous ones, and no two of these new intersection
points coincide. 

\end{enumerate}

We first claim that each connected component of $M \sm (\gamma \cup
\lambda)$ is a rectangle. Take a segment $I \subset \gamma$ which is
bounded by either a singularity and the first point where a
$\lambda_\ell$ comes down to $\gamma$, or by two consecutive points at
which $\lambda_\ell$'s come down to $\gamma$. Consider the union of
vertical leaves which begin at points of $I$ and move upwards. These
form a strip, which near $\gamma$ is bounded on both sides by
segments in $\lambda$. By choice of the $\lambda_j$ all these leaves
all hit $\gamma$ again `at 
the same time', i.e. none of these leaves hit a singularity or a
terminal point before returning to $\gamma$. The construction also
ensures that the segments from $\lambda$ on both sides of the strips
also extend until the next intersection with $\gamma$. Thus the strip
is a rectangle bounded by $\gamma \cup \lambda$. 

We now claim that this partition into rectangles has the Markov
property. Since $\varphi^{-1}$ fixes each singularity and maps each critical
leaf to itself, and since it contracts the horizontal direction,
$\gamma$ is mapped into itself. This implies that lower boundaries 
of rectangles map into lower boundaries under
$\varphi^{-1}$. For identical reasons each $\lambda_j$ is mapped into
itself by $\varphi$. 

Now note that the endpoints of lower edges of rectangles
on $\gamma$ are the terminal point $p$, the singularity, and one
additional point for each downward pointing edge in $\lambda$. In
total we have 
$2+k/2$ endpoints which gives $1+k/2$ rectangles. In particular the
number of rectangles depends only on $\MM$.
\end{proof}

\end{appendix}

\end{document}